\documentclass[12pt]{amsart}
\usepackage{amscd}
\usepackage{amsmath}
\usepackage{ulem}
\usepackage{amssymb}
\usepackage{amsthm}
\usepackage[usenames,dvipsnames]{color}

\begin{document}

\newtheorem{thm}{Theorem}
\newtheorem{prop}[thm]{Proposition}
\newtheorem{llama}[thm]{Lemma}
\newtheorem{sheep}[thm]{Corollary}
\newtheorem{deff}[thm]{Definition}
\newtheorem{fact}[thm]{Fact}
\newtheorem{example}[thm]{Example}
\newtheorem{slogan}[thm]{Slogan}
\newtheorem{remark}[thm]{Remark}
\newtheorem{quest}[thm]{Question}
\newtheorem{zample}[thm]{Example}

\newcommand{\sthat}{\hspace{.1cm}| \hspace{.1cm}}
\newcommand{\id}{\operatorname{id} }
\newcommand{\acl}{\operatorname{acl}}
\newcommand{\dcl}{\operatorname{dcl}}
\newcommand{\irr}{\operatorname{irr}}
\newcommand{\aut}{\operatorname{Aut}}
\newcommand{\fix}{\operatorname{Fix}}

\newcommand{\oo}{\mathcal{O}}
\newcommand{\aaa}{\mathcal{A}}
\newcommand{\mm}{\mathcal{M}}
\newcommand{\curg}{\mathcal{G}}
\newcommand{\bbf}{\mathbb{F}}
\newcommand{\A}{\mathbb{A}}
\newcommand{\R}{\mathbb{R}}
\newcommand{\Q}{\mathbb{Q}}
\newcommand{\C}{\mathbb{C}}
\newcommand{\cc}{\mathcal{C}}
\newcommand{\dd}{\mathcal{D}}
\newcommand{\N}{\mathbb{N}}
\newcommand{\Z}{\mathbb{Z}}
\newcommand{\cF}{\mathcal F}
\newcommand{\cB}{\mathcal B}
\newcommand{\cU}{\mathcal U}
\newcommand{\cV}{\mathcal V}
\newcommand{\cG}{\mathcal G}
\newcommand{\cD}{\mathcal D}
\newcommand{\curly}{\mathcal{C}}
\newcommand{\durly}{\mathcal{D}}
\newcommand{\fff}{\mathcal{F}}
\newcommand{\GGG}{\mathcal{G}}
\newcommand{\calc}{\mathcal{C}}

\newcommand{\alice}[1]{{\color{blue} \sf $\clubsuit\clubsuit\clubsuit$ Alice: [#1]}}
\newcommand{\ramin}[1]{{\color{red}\sf $\clubsuit\clubsuit\clubsuit$ Ramin: [#1]}}

\newcommand{\Fmodtor}{F^\times / \mu(F)}

\title[Hereditarily Irreducible Polynomials]{Multiplicative groups of fields and hereditarily irreducible polynomials}
\author{Alice Medvedev}
\address{A. M.: Department of Mathematics, The City College of New York, NAC 8/133, Convent Ave at 138th Street
New York, NY 10031}
\email{amedvedev@ccny.cuny.edu}
\thanks{The first author is partially supported by the Simons Foundation Collaborative Grant (Award number 317672), by NSF DMS-1500976, and by CUNY CIRG Project number 2248.
The second author is partially supported by the Simons Foundation Collaborative Grant (Award number 245977) and by National Science Foundation Research Training Grant Award number DMS-1246844.}
\author{Ramin Takloo-Bighash}
\address{R. T.-B.: Department of Mathematics, Statistics, and Computer Science, University of Illinois at Chicago, 851 S Morgan St (M/C 249), Chicago, IL 60607}
\email{rtakloo@uic.edu}
\begin{abstract}
In this paper we explore the concept of {\em good heredity} for fields from a group theoretic perspective. Extending results from \cite{alice}, we show that several natural families of fields are of good heredity, and some others are not. We also construct several examples to show that various wishful thinking expectations are not true.
\end{abstract}

\maketitle

\tableofcontents

\section{Introduction}

In this paper we investigate a few properties of fields closely related to the freeness of their multiplicative groups. The first author stumbled upon one of them in her investigation of divisibility of quasiendomorphisms of abelian varieties in \cite{alice}. This property of a field $F$, which we call \emph{good heredity}, deals with (ir)reducibility of $P(x^n)$ for a polynomial $P(x)$ over $F$, as $n$ varies,
The other two are properties of the multiplicative group of the field, \emph{probably} slightly weaker than freeness. We were surprised to discover that the two are not equivalent to each other, nor to good heredity (see Examples \ref{one-converse-fail} and \ref{converse-fail2} and Section 6.2). One advantage of our good heredity over all these related properties of the multiplicative group is that good heredity passes to finite extensions of fields (see Theorem \ref{gher-finext}), while the various properties of the multiplicative group do not (see example in \S \ref{rottenroots} and \S \ref{freemodtor-zample}).

\begin{deff}
 For a field $F$,  \begin{itemize}
 \item let $F^\times$ denote the multiplicative group of $F$, an abelian group;
 \item let $\mu(F)$ denote the group of roots of unity on $F$, i.e. the torsion subgroup of $F^\times$;
 \item and let $\Fmodtor$ denote their quotient, a torsion-free abelian group.
 \end{itemize}
\end{deff}

May \cite{may-1972} shows that any locally cyclic abelian group can show up as a direct summand of $F^\times$.
May \cite{may-1980} proves that for many interesting fields, $\Fmodtor$ is a free abelian group; and constructs an example showing that this property (freeness of $\Fmodtor$) does not pass to finite extensions, even in the very tame setting of algebraic extensions of $\mathbb{Q}$.

\begin{deff} \label{introdef}
An abelian group $G$ is \emph{rootless} if for any non-torsion element $a \in G$, the divisible hull inside $G$ of the subgroup of $G$ generated by $a$ is free.
We call $G$ \emph{free modtor} if the quotient of $G$ by its torsion subgroup is free Abelian.

A field $F$  is \emph{rootless} if the group $F^\times$ is rootless.

A field $F$  is \emph{rootless modtor} if the group $\Fmodtor$ is rootless.
\end{deff}

It is easy to see that if $\Fmodtor$ is free, then $F$ is rootless modtor; and if $F$ is rootless modtor, then $F$ is rootless. We show that the converses of both statements are false with Example Sections \ref{modtormodtor} and \ref{rootless-notrootless}.

We show that, like freeness of $\Fmodtor$, these two properties of fields do not pass to finite extensions (see Example 6.1). However, they shed light on the third, more complicated property we call ``good heredity'' (see Definition \ref{goodfieldef}), which does pass up to finite extensions (Theorem \ref{gher-finext}). This good heredity was the original motivation for this investigation. We show that good heredity of $F$ implies that every finite extension of $F$ is rootless (Theorem \ref{gher-finext} and Proposition \ref{shepherding}), and is implied by every finite extension of $F$ being rootless modtor (Proposition \ref{shepherding}). We do not know whether converses of these hold; if the following is true, all three are equivalent.

\begin{quest} Suppose that every finite extension of a field $F$ is rootless; does it follow that $F$ is rootless modtor?\\ In particular, does the rootless but not rootless modtor field constructed in Section 6.2 have a finite extension that is not rootless?
\end{quest}


In \cite{alice}, the first author needs to understand finite-to-finite group correspondences (i.e. quasiendomorphisms) of cartesian powers of an elliptic curve $E$. These can be encoded by matrices over $F$, the field of fractions of the ring of endomorphisms of $E$. Good heredity of the number field $F$ is used to analyze the characteristic polynomials of these matrices. While the model-theoretic goals of \cite{alice} are beyond the scope of this paper, they do suggest a desirable generalization, from elliptic curves to arbitrary simple abelian varieties.

\begin{quest}
Do quasiendomorphism division rings of higher dimensional abelian varieties have good heredity? and what does that even mean when multiplication is not commutative?\end{quest}

This paper includes a few model-theoretic side comments for the initiated, without defining or explaining any of the terminology; rest assured that none of these side comments are used in the main flow of logic.


Our results go well beyond the model-theoretic motivations of \cite{alice}.
We show that the following fields have good heredity: \begin{itemize}
\item $\mathbb{F}_p^{alg}$ for any prime $p$;
\item any finitely generated field (Corollary \ref{fingensheep});
\item the maximal abelian extension of any number field (Corollary \ref{maysheep}, leveraging May's results from \cite{may-1980});
\item any extension of $\mathbb{Q}$ by a set of algebraic elements of bounded degree (Corollary \ref{maysheep}, leveraging May's results from \cite{may-1980});
\item any subfield of a field of good heredity (Remark \ref{whole-field-ezis}); and
\item any finitely generated field extension of a field of good heredity (Proposition \ref{goodfields}).
\end{itemize}

 All of these except for $\mathbb{F}_p^{alg}$ are very far from being algebraically closed. Indeed, the multiplicative group of an algebraically closed field is divisible, so with the exception of $\mathbb{F}_p^{alg}$ where all elements are roots of unity, no algebraically closed field is rootless, so none have good heredity. More generally, we show that no local field is rootless, so none of those have good heredity (Corollary \ref{local-hensel}).

\begin{quest}
Are there are any pseudofinite fields of good heredity? More generally, any PAC fields of good heredity? Morally, how far does a field have to be from ACF to have a chance of good heredity?
\end{quest}

The last section of this paper contains nasty counterexamples to two things we tried very hard to prove: that rootless and rootless modtor are the same, and that this property passes to finite field extensions.


\

{\em Terminology Warning 1.}  Beware that the term ``hereditarily irreducible polynomial'' has been used to mean something else (though related) in \cite{ruschitver, abhyrub, angmu}. In these references a polynomial $P(y_1, \dots, y_n)$ is called hereditarily irreducible, if for all non-constant single variable polynomials $f_1, f_n$ the polynomial $P(f_1(y_1), \dots, f_n(y_n))$ is irreducible. This is a far more restrictive condition, and much harder to check for specific polynomials, or classes thereof. It is, for example,  an already non-trivial result from \cite{ruschitver} that if $f(x)$ is any square free polynomial of degree exceeding one, then the two variable polynomial $P(x, y) = y f(x) + 1$ is hereditarily irreducible.

\

{\em Terminology Warning 2}: In this article \emph{roots} refers to roots of elements of a field, e.g. roots of unity, or the second root of $2$. When talking about polynomial, we will use \emph{zero} to mean the root of the polynomial.

\section{polynomials}

\begin{deff}
 A \emph{hereditary factor} of a polynomial $P \in F[x]$ over a field $F$ is any factor of $P(x^n)$ for some $n \in \mathbb{N}$.

 A polynomial $P \in F[x]$ over a field $F$ is \emph{hereditarily irreducible over $F$} if for every $n \in \mathbb{N}$, the polynomial $P(x^n)$ is irreducible over $F$.

 A polynomial $P \in F[x]$ has \emph{good heredity} if $P(x^n)$ factors into hereditarily irreducible factors over $F$, for some positive integer $n$.
\end{deff}

One interesting question, intimately related to the arithmetic of the field $F$, is the classification of hereditarily irreducible polynomials in $F[x]$. The following question seems like an interesting question:
\begin{quest}
Fix a field $F$. What polynomials over $F$ are hereditarily irreducible over $F$?
\end{quest}

\begin{remark}
Hereditarily irreducible polynomials are in particular irreducible. The inverse is clearly false: polynomials $x$ and $x-1$ are not hereditarily irreducible over any field;  more generally, no polynomial whose zeroes are roots of unity is hereditarily irreducible.
\end{remark}

It is not clear what criteria one can formulate to guarantee a polynomial be hereditarily irreducible, but the following is an interesting example.

\begin{llama} If $R$ is a UFD, Eisenstein polynomials are hereditarily irreducible over the fraction field of $R$.
\end{llama}

\begin{proof} Let $R$ be an integral domain, and $P \in R[x]$ an Einsenstein polynomial with respect to a prime ideal $p$. Then for any $n \in \N$, $P(x^n)$ is Eisenstein with respect to $p$. If $R$ is also a UFD, then Eisenstein polynomials are irreducible over $R$ and, by Gauss's Lemma, over its field of fractions. \end{proof}

It can be enlightening to arrange hereditary factors of a polynomial into a tree. To avoid spurious branching due to constant factors, we require polynomials to be monic--an entirely harmless assumption over a field. We use this construction in the proof of Theorem \ref{alicelemma}.

\begin{deff}
For a monic irreducible polynomial $P \in F[x]$, let $T_0(P, F)$ be the following tree.
The nodes on the $n$-th level of $T_0(P, F)$ are the monic irreducible factors of $P(x^{n!})$ over $F$. 
The partial order of the tree is given by divisibility: a factor $R(x)$ of $P(x^{(n+1)!})$ lies above a factor $Q(x)$ of $P(x^{n!})$ whenever $R(x)$ divides $Q(x^{n+1})$. \end{deff}

If $P$ is not irreducible, the same construction yields a forest instead of a tree: $T_0(P,F)$ has several nodes on the lowest level.

To construct a similar object for a non-monic polynomial $P$, let $a$ be the leading coefficient of $P$, and let $\tilde{P}:= \frac{1}{a} P$, a monic polynomial. Finally, pick one branch of $T_0(\tilde{P}, F)$ and multiply all nodes on that branch by $a$. The arbitrary choice of the branch is the reason we stick to monic polynomials in the definition. Alternately, we could take the nodes to be equivalence classes of irreducible polynomials under the equivalence relation of ``non-zero multiple''.

Since $P(x)$ is irreducible, it has no repeated roots in $F^{alg}$. As long as $P(x) \neq x$, it follows that $P(x^n)$ also has no repeated roots in $F^{alg}$; and then $P(x^n)$ has no repeated factors in $F[x]$.
Thus, by unique factorization in $F[x]$, the product of all the nodes on $n$-th level of $T_0(P, F)$ is precisely $P(x^{n!})$. The polynomial $P(x) = x$ is ignored in this paper, as it is clearly not hereditarily irreducible; the particularly fastidious reader is invited to trace this exception through the rest of this paper to see that it does not break any of our proofs.

A factor $Q(x)$ of $P(x^{n!})$ is hereditarily irreducible if and only if there are no splits above it in this tree $T_0(P,F)$; the sole exception is $P = x$, for which every level of the tree contains nothing but $x$.

\begin{deff}
Let $T(P,F)$ be the tree obtained from $T_0(P,F)$ by trimming all nodes above hereditarily irreducible factors.
\end{deff}

\begin{llama} \label{fintreellama}
The polynomial $P \in F[x]$ has good heredity if and only if $T(P,F)$ is finite.
\end{llama}

\begin{proof}
This finitely-branching tree is infinite if and only if it has an infinite branch. That is, if there are integers $1= n_0 < n_1 < n_2 < n_3 < \ldots $ (all factorials, with $n_i$ dividing $n_{i+1}$ for each $i$) and irreducible factors $Q_i(x) \in F[x]$ of $P(x^{n_i})$ such that $Q_{i+1}(x)$ properly divides $Q_i(x^{\frac{n_{i+1}}{n_i}})$ for each $i$.
\end{proof}

\begin{llama} \label{treellama}
If the polynomial $P \in F[x]$ has good heredity, then for any $M \in \N$ there are only finitely many polynomials of degree less than or equal to $M$ in $T_0(P,F)$. \end{llama}

\begin{proof} The only polynomials lying above a hereditarily irreducible factor $Q(x)$ of $P(x^{n!})$ in $T_0(P,F)$ are of the form $Q(x^m)$ for increasing $m$, so only finitely many of them will have degree less than $M$. \end{proof}

\section{(in)divisibility in the multiplicative group}

\subsection{Roots}

In this section we explore the connections between hereditary (ir)reducibility in $F[x]$ and (in)divisibility in the multiplicative group of $F$, that is, lack of roots. The problem of understanding the structure of the multiplicative group of an arbitrary field, or the classification of those Abelian groups with locally cyclic torsion subgroups which happen to be (isomorphic to) multiplicative groups of fields is very non-trivial, and despite non-trivial contributions by many mathematicians over the last half a century, there are many open questions in this area. For an old survey of results and review of the history of these results, see \cite{magicbook}, Chapter 4.
\begin{deff}
An element $f$ of a field $F$ is \emph{very rootless} if there is no $g \in F$ and $n \in \N$ with $n \geq 2$ and $f = g^n$.
An element $f$ of a field $F$ is \emph{very rootless modtor} if there is no root of unity $\zeta \in F$, element $g \in F$ and integer $n \geq 2$ such that $f = \zeta g^n$.
\end{deff}
For example in the field $\Q$, the element $-4$ is very rootless, but not very rootless modtor.  An element $f$ of the field $F$ is very rootless modtor if the image of $f$ in the quotient of the multiplicative group of $F$ by its torsion (i.e. by the group of roots of unity) has no proper roots.

\

Note that $a$ is very rootless (modtor) if and only if polynomials $x^n - a$ ($x^n - \zeta a$ for roots of unity $\zeta$) have no zeros in $F$ for $n \geq 2$.

\begin{llama} \label{modtorllama} If $F$ contains all roots of unity then an element $f \in F$ is rootless, resp. very rootless, if and only if it is rootless modtor, resp. very rootless modtor.\end{llama}
\begin{proof}
If $F$ contains all roots of unity, then any root of unity $\zeta$ has an $n$-th root $\xi \in F$ for any $n$, so that
 $b \in F$ is a solution of $x^n = a$ if and only if $\xi b$ is a solution of $x^n = \zeta a$.
\end{proof}

\begin{prop} \label{rootless-linear}
 If an element $a$ is very rootless modtor, then $(x-a)$ is hereditarily irreducible over $F$.
 If $(x-a)$ is hereditarily irreducible over $F$, then $a$ is very rootless.
\end{prop}

\begin{proof} If $Q(x)$ is an irreducible factor of $x^n -a$ of degree $m \lneq n$, then every zero $b$ of $Q$ has $b^n =a$, and the product $c$ of all $m$ of them is in $F$ and satisfies $c^n = a^m$. Let $k := \gcd (m,n)$ so that $(c^{n'})^k = (a^{m'})^k$, so $(c^{n'}) = \zeta (a^{m'})$ for some $k$th root of unity $\zeta$. Since $m \lneq n$, it must be that $n' \gneq 1$. Since $m'$ and $n'$ are relatively prime, there are integers $\alpha$ and $\beta$ such that $\alpha m' + \beta n' =1$. Now
 $$a = a^{\alpha m' + \beta n'} = a^{\alpha m'}a^{\beta n'} = \zeta^{-\alpha} (c^{n'})^\alpha a^{\beta n'}
 = \xi (c^\alpha a^{\beta})^{n'}$$
 makes $a$ a non-trivial power modulo torsion, so $a$ cannot be very rootless modtor.

If $a = c^n$, then $x-c$ divides $x^n - a$.
\end{proof}

\begin{zample} \label{one-converse-fail}
The converses of the two statements can fail.
On one hand, $x^n +4$ is irreducible over $\mathbb{Q}$ for all $n$, but $-4$ is not very rootless modtor.
On the other hand, $-2$ is very rootless in $F := \mathbb{Q}(\sqrt[4]{2})$, but $x^4 + 2$ factors over $F$ as
 $(x^2 + \alpha^2 + \alpha^3x)(x^2 + \alpha^2 - \alpha^3x)$ where $\alpha^4 = 2$.
\end{zample}

\begin{sheep} \label{withroots1}
If $F$ contains all roots of unity and $a \in F$, the following are equivalent.\begin{itemize}
\item $a$ is very rootless modtor.
\item $(x-a)$ is hereditarily irreducible.
\item $a$ is very rootless.
\end{itemize}
Also, none of those hold if $a$ is a root of unity.
\end{sheep}

\begin{proof} If $F$ contains all roots of unity, then any root of unity has any root in $F$. Thus, very rootless and very rootless modtor are the same notion. \end{proof}

\subsection{Descending chains}

\begin{deff}
An element $a \in F$ that is not a root of unity is \emph{rootless} if cofinitely many of the equations $x^n = a$ have no solutions in $F$.

An element $a \in F$ that is not a root of unity is \emph{rootless modtor} if for cofinitely many $n$ the equation $x^n = \zeta a$ has solutions in $F$ for any root of unity $\zeta \in F$.
\end{deff}

\begin{remark} \label{rootsrk}
An element $f \in F$ is rootless if and ony if the set \\ $\{ g \in F : g^n = f\mbox{ for some } n \in \mathbb{N} \}$ of roots of $f$ is finite.

An element $f$ of the field $F$ is rootless modtor if the image of $f$ in the quotient $F^\times /\mu(F)$ of the multiplicative group of $F$ by its torsion  has only finitely many proper roots. Equivalently, the subgroup $R(f,F)$ of $(\mathbb{Q}, +)$ defined by
$$R(f,F) := \{ q \in \mathbb{Q} \sthat \bar{f}^q \in F/\mu \}$$
is cyclic. (Here, $\bar{f}$ is the image of $f$ in $F/\mu$.) Since $F^\times /\mu(F)$ is torsion-free, roots in it are well defined when they exist, so $\bar{f}^q$ is well-defined.
\end{remark}

\begin{llama} \label{roothered}
Fix a field $F$ and an element $a \in F$ that is not a root of unity.

If $a$ is rootless modtor, then $P(x) := x-a$ has good heredity.

If $P(x)$ has good heredity, then $a$ is rootless.
\end{llama}

\begin{proof}
The contrapositive of the second claim is easiest to prove. Suppose that there is a sequence of natural numbers
$\ell_1 < \ell_2 < \ell_3 < \dots$, and elements $\beta_i \in F^\times$ with
$$ a = \beta_i^{\ell_i} $$ for each $i$.
Suppose that $P(x^n)$ is a product of hereditarily irreducible polynomials $P_i$ for some $n$. Next,
$$ P(x^{n \ell_i}) = x^{n \ell_i} - a = x^{n \ell_i} - \beta_i^{\ell_i} = (x^n- \beta_i) P_i(x).$$
Since $a$ is not a root of unity, all $\beta_i$ are distinct, and so $(x^n- \beta_i)$ are infinitely many distinct polynomials of the same degree in the finitely many trees $T_0(P_i,F)$, contradicting Lemma \ref{treellama} in at least one of the trees.

Now, for the contrapositive of the first claim,
Suppose that $(x- a)$ does not have good heredity over $F$, so there is an infinite branch in $T(P,F)$:\begin{enumerate}
\item integers $1= n_0 < n_1 < n_2 < n_3 < \ldots$ and
\item irreducible polynomials $Q_i(x) \in F[x]$ of degree $k_i$, such that:
\item\label{[*]} all $n_i$ are factorials, and $n_i$ divides $n_{i+1}$ for each $i$;
\item\label{[**]} $Q_i(x)$ divides $P(x^{n_i})$; and
\item\label{[***]} $Q_{i+1}(x)$ properly divides $Q_i(x^{\frac{n_{i+1}}{n_i}})$ for each $i$.
\end{enumerate}

From \eqref{[**]}, $Q_i(x)$ divides $(x^{n_i}- a)$, so the $n_i$th power of any zero of $Q_i$ is $a$.
Let $a_i$ be the product of the $k_i$ zeros of $Q_i$; then
 $$a_i^{n_i} = a^{k_i}.$$
This $a_i$ is in $F$ because $(-1)^{k_i}a_i$ is the constant coefficient of $Q_i$.
Let $b_i$ and $b$ be images of $a_i$ and $a$ in the quotient $F/\mu$ of the multiplicative group of $F$ by its torsion, i.e. roots of unity. Of course, we still have
 $$b_i^{n_i} = b^{k_i}.$$
In the terminology of Remark \ref{rootsrk}, we get $k_i/n_i \in R(a,F)$ for each $i$.
Since \eqref{[***]} implies that $\frac{k_{i+1}}{n_{i+1}} \lneq \frac{k_i}{n_i}$ for each $i$, this makes $R(a,F)$ not cyclic, and, therefore, $a$ not rootless modtor.
\end{proof}

Again, converses may fail.

\begin{zample} \label{converse-fail2}
Consider $F_n := \mathbb{Q}(\sqrt[2^n]{17})$ and let $a := -17$.
Note that, $a$ is even very rootless: as $F_n$ is embeddable in the reals, $a$ has no even-degree roots; and for any odd $p$, the degree of $\mathbb{Q}(\sqrt[p]{-17})$ over $\mathbb{Q}$ is odd, so $\mathbb{Q}(\sqrt[p]{-17})$ cannot be a subfield of $F_n$. Of course, $(-1)(-17) = 17$ is not rootless in $F$. We suspect that $x+17$ is hereditarily irreducible over $F_n$, but we do not know how to prove this.  \end{zample}

An analog of Corollary \ref{withroots1} holds: if $F$ has enough roots of unity in it, then $a$ is rootless if and only if it is rootless modtor, so also if and only if $x-a$ has good heredity.

\subsection{Higher-degree polymonials.}
We start this subsection with a remark.
\begin{remark} \label{easyrk}
A polynomial $P(x)$ has good heredity if and only if $P(x^n)$ has good heredity for some/all $n$.

By Unique Factorization in polynomial rings over fields, any product of hereditarily irreducible polynomials has good heredity.

It is easy to see that if $F \leq E$ and $P \in F[x]$ has good heredity over $E$, then $P$ also has good heredity over $F$. (See Section \ref{logicsec} below for an overly technical explanation of this.)

\end{remark}

The next result was the original motivation for this work, Lemma 4.15 of \cite{alice}. Much of its original proof has already been presented in bits and pieces earlier in this paper.

\begin{thm} \label{alicelemma}
Let $F$ be a field and let $P(x) \in F[x]$ be a separable irreducible polynomial whose zeros in $F^{alg}$ are not roots of unity. Suppose for every root $a \in F^{alg}$ of $P$, $a$ is rootless modtor in $F(a)$. Then $P$ has good heredity over $F$.\end{thm}

\begin{proof}
 By Lemma \ref{fintreellama}, it suffices to show that the tree $T(P,F)$ is finite. Suppose towards contradiction that it is not. This infinite, finitely-branching tree $T$ must have an infinite chain: integers $1= n_0 < n_1 < n_2 < n_3 < \ldots $ (all factorials, with $n_i$ dividing $n_{i+1}$ for each $i$) and irreducible factors $Q_i(x) \in F[x]$ of $P(x^{n_i})$ such that $Q_{i+1}(x)$ properly divides $Q_i(x^{\frac{n_{i+1}}{n_i}})$ for each $i$.

 Since $P$ is irreducible, any two roots of $P$ in $F^{alg}$ are conjugate by an automorphism $\rho$ of $F^{alg}$ over $F$. Since $\rho$ fixes the coefficients of $Q_i$, each root of $P$ has the same number $k_i \leq n_i$ of $n_i$th roots that are also of roots of $Q_i$. Since $Q_{i+1}(x)$ properly divides $Q_i(x^{\frac{n_{i+1}}{n_i}})$ for each $i$, we must have $\frac{k_{i+1}}{n_{i+1}} \lneq \frac{k_i}{n_i}$ for each $i$.

 Fix a root $a \in F^{alg}$ of $P$, let $\tilde{F} := F(a)$, and let $\tilde{P}(x) := x-a$, an irreducible polynomial in $\tilde{F}[x]$. One of the hypotheses of this theorem is that $a$ is rootless modtor in $\tilde{F}$, and it follows by Proposition \ref{rootless-linear} that $\tilde{P}$ is hereditarily irreducible over $\tilde{F}$.
 The $k_i$ $n_i$th roots of $a$ which are also roots of $Q_i$ form a Galois orbit over $\tilde{F}$ and thus correspond to a factor $\tilde{Q}_i \in \tilde{F}[x]$ of $\tilde{P}(x^{n_i}) = x^{n_i} -a$, of degree $k_i$.
 Now these $\tilde{Q}_i$ form an infinite branch in $T(\tilde{P}, \tilde{F})$ contradicting hereditary irreducibility of $\tilde{P}$ over $\tilde{F}$ via Lemma \ref{fintreellama} again.
\end{proof}

\section{Fields of good heredity: Basic properties}

It is clear that a polynomial which is divisible by $x$ cannot be hereditarily irreducible.

\begin{deff} \label{goodfieldef}
A field $F$ \emph{has good heredity} if every polynomial $P \in F[x]$ none of whose roots in $F^{alg}$ are roots of unity or zero has good heredity over $F$.

A field $F$ is \emph{rootless} if any $f \in F$ that is not a root of unity is rootless in $F$.

A field $F$ is \emph{rootless modtor} if any $f \in F$ that is not a root of unity is rootless modtor in $F$.
\end{deff}

\begin{remark} \label{whole-field-ezis}
These three properties of fields (rootless, rootless modtor, and good heredity) pass to subfields.
\end{remark}

\begin{proof}
The failure of each of these properties passes up from subfields because it is witnessed by an infinite collection of elements of the field satisfying some equations and failing to satisfy others. The witnesses are still there in the bigger field, and the two fields agree on whether the equations are satisfied. \end{proof}

We have already shown the following.

\begin{prop} \label{shepherding}
If $F$ is rootless modtor, then all linear polynomials in $F$ have good heredity over $F$. If all linear polynomials in $F$ have good heredity over $F$, then $F$ is rootless.

If $F$ contains all roots of unity, or only finitely many roots of unity, then these three properties are equivalent.

If all finite extensions of $F$ are rootless modtor, then $F$ has good heredity.
\end{prop}

\begin{proof}
The first statement is an immediate consequence of Lemma \ref{roothered}.

The second statement follows directly from Lemma \ref{modtorllama} for fields with all roots of unity. If $F$ only has finitely many roots of unity and $f \in F$ is not rootless modtor, then infinitely many equations $x^n = \zeta_n f$ have solutions in $F$ - but then, as infinitely many integers $n$ try to fly into finitely many roots of unity $\zeta$, some $\zeta$ must appear infinitely often, making $\zeta f$ not rootless in $F$.

The last statement is an immediate consequence of Theorem \ref{alicelemma}.
\end{proof}

\begin{sheep}\label{fund-cor}
  If the multiplicative group of every finite extension of the field $F$ is free modulo torsion, then $F$ is of good heredity.
\end{sheep}

\begin{proof} The first sentence after Definition \ref{introdef}.\end{proof}

\begin{prop}\label{rootless-trans} A purely transcendental extension of a rootless (resp. rootless modtor) field is rootless (resp. rootless modtor).
\end{prop}

\begin{proof}
Lemma 10.2 on page 503 of \cite{magicbook} states that for any purely transcendental extension $E$ of $F$, there is some free Abelian group $A$ such that $E^\times \simeq F^\times \times A$.\end{proof}

The situation for algebraic extensions is far more complicated, as demonstrated by the existence of examples of the sort presented in \S\ref{rottenroots}. However, good heredity passes to finite extensions, making it a better property than rootlessness.

\begin{thm} \label{gher-finext}
If a field $F$ has good heredity and $E \geq F$ is a finite field extension of $F$, then $E$ also has good heredity. \end{thm}

\begin{proof}
Since good heredity of polynomials passes to subfields (Remark \ref{easyrk}), it suffices to consider the case where $E/F$ is Galois. Let $\tilde{P}(x) \in E[x]$ be an irreducible polynomial whose zeros are not roots of unity; we need to show that $\tilde{P}(x)$ has good heredity over $E$.

Suppose that $\tilde{P}$ does not have good heredity over $E$; then, as in the proof of Lemma \ref{roothered}, there are \begin{enumerate}
\item integers $1= n_0 < n_1 < n_2 < n_3 < \ldots$ and
\item irreducible polynomials $\tilde{Q}_i(x) \in F[x]$ of degree $k_i$, such that:
\item all $n_i$ are factorials, and $n_i$ divides $n_{i+1}$ for each $i$;
\item $\tilde{Q}_i(x)$ divides $\tilde{P}(x^{n_i})$; and
\item $\tilde{Q}_{i+1}(x)$ properly divides $\tilde{Q}_i(x^{\frac{n_{i+1}}{n_i}})$ for each $i$.
\end{enumerate}
Let $G$ be the Galois group of $E$ over $F$, and let
 $$P(x) := \prod_{\sigma \in G} \sigma(\tilde{P}) \in F[x]\mbox{, and}$$
 $$Q_i(x) := \prod_{\sigma \in G} \sigma(\tilde{Q}_i) \in F[x].$$
Now the same integers $n_i$ together with polynomials $Q_i$ witness that $P$ does not have good heredity over $F$.
Since any zero $b \in F^{alg}$ of $P$ is a Galois conjugate of some zero $\tilde{b}$ of $\tilde{P}$, none of the zeros of $P$ are roots of unity. A few more steps as in Remark \ref{easyrk} obtain irreducible $Q_i$ finishing the proof.
\end{proof}

\section{A model-theorist's musings.} \label{logicsec}

Here the first author collects a few observation on the logical complexity of rootlessness and heredity; they are mostly irrelevant to the rest of the paper.

For a subset $S \subset \mathbb{N}$, let $rooty_S(x)$ be the type
$$rooty_S(x) := \{ \exists y\, y^\ell = x \sthat \ell \in S \} \cup \{ x^\ell \neq 1 \sthat \ell \in \mathbb{N}^+ \}.$$
An element $a$ of a field $F$ is rootless if and only if it does not realize $rooty_S$ for any infinite $S$. In particular, being rootless is an $L_{\kappa, \omega}$-universal property (for $\kappa = 2^\omega$), so it passes to subfields of $F$ containing $a$. Similarly, A field $F$ is rootless if and only if it omits the types $rooty_S$ for all infinite $S$; so this property passes to all subfields of $F$.

Similarly, failure to be rootless modtor is witnessed by realizing the type
$$\{ \exists y\, \exists z\, y^m = z x \mbox{ and } z^n = 1  \sthat (m,n) \in S \} \cup \{ x^\ell \neq 1 \sthat \ell \in \mathbb{N}^+ \}$$
For some function $S \subset \mathbb{N} \times \mathbb{N}$ with infinite domain.

It is clear, but cumbersome to write down, that hereditary irreducibility and good heredity are also negations of large disjunctions of first-order existential types, both for individual polynomials and for entire fields; they again pass to subfields.

All these properties are very far from being first order, as witnessed by the following proposition.


\begin{prop}
Any infinite field has an elementary extension that is not rootless and, therefore, is not of good heredity.
\end{prop}

\begin{proof}
Consider the type
 $$p(x) := \{ \exists y\, y^\ell = x \sthat \ell \in \mathbb{N}^+ \} \cup \{ x^\ell \neq 1 \sthat \ell \in \mathbb{N}^+ \}.$$
It suffices to produce an elementary extension of the original field $F$ with a realization of $p$. That is, it suffices to show that $p$ is consistent with the theory of $F$. That is, it suffices to show that any finite subset of $p$ is realized in $F$, or in some elementary extension of $F$.
 Any finite subset of $p$ is contained in
 $$p_N(x) := \{ \exists y\, y^\ell = x \sthat 1 \leq \ell \leq N \} \cup \{ x^\ell \neq 1 \sthat 1 \leq \ell \leq N \}$$
for some $N$.
 Since $F$ is infinite, it has an elementary extension $F^+$ containing an element $a$ that is not a root of unity.
 Now $a^{N!}$ realizes $p_N$.
\end{proof}

\section{Fields of good heredity: Examples} \label{pigsfly}

Here we use other people's work to find fields with good heredity, mostly by finding fields all of whose finite extensions are rootless modtor. We start with some easy observations we could've made long ago.

\begin{prop} \label{goodfields}
\begin{enumerate}
\item The field $\bbf_p^{alg}$ is of good heredity. These fields are the only algebraic closed fields of good heredity.
\item  All global fields in any characteristic are rootless, rootless modtor, and of good heredity.

\item\label{hered-trans} A finitely generated extension of a field of good heredity is of good heredity.   \end{enumerate} \end{prop}

\begin{proof} The field $\bbf_p^{alg}$ is vacuously of good heredity, as every non-zero element is a root of unity.
The multiplicative group of an algebraically closed field is divisible.  By Dirichlet's Unit Theorem, the multiplicative group of any global field and all of its finite extensions are free modulo torsion. Corollary \ref{fund-cor} implies the statement about global fields.  Since every finitely generated extension of a field is an algebraic extension of a purely transcendental extension, Theorem \ref{gher-finext} implies that in order to prove statement \ref{hered-trans} we may assume that the extension is purely transcendental. But for a purely transcendental extension $F(S)$ of $F$, a reducibility statement in $F(S)[x]$ gives rise to a reducibility statement in $F[x]$ by specialization.
\end{proof}

\begin{sheep} \label{fingensheep}
Any field finitely generated over $\mathbb{Q}$ or over any subfield of $\bbf_p^{alg}$ has good heredity.\end{sheep}
\begin{proof}
By Proposition \ref{goodfields}, $\mathbb{Q}$ and subfields of $\bbf_p^{alg}$ have good heredity.
Now induct on the number of generators, using Lemma \ref{gher-finext} deals with algebraic extensions and Remark \ref{goodfields} again to deal with transcendental ones. \end{proof}

\begin{sheep} \label{local-hensel}
Local fields are not rootless and, therefore, are not of good heredity.
\end{sheep}
\begin{proof}
Archimedean local fields $\mathbb{R}$ and $\mathbb{C}$ are clearly not of good heredity.  So, let $F$ be a non-archimedean local field of residue characteristic $p$. Let $\oo$ be the ring of integers of $F$, and $\mathfrak{p}$ the prime ideal.  Then let $\alpha \in \oo$ be an integer satisfying
$$
\alpha \equiv 1 \mod \mathfrak{p}.
$$
We also assume that $\alpha$ is not a root of unity. Let $r$ be a natural number not divisible by $p$. Now consider the equation $x^r = \alpha$. This equation reduces to $x^r \equiv 1 \mod \mathfrak{p}$, which is clearly solvable. Furthermore, $(x^r)' = r x^{r-1}$ evaluated at $\alpha$ is congruent to $r$ modulo $\mathfrak{p}$ which
is by assumption non-zero.  Hensel's lemma for non-archimedean local fields shows that the equation $x^r = \alpha$ is solvable in $F$. Now Lemma  \ref{roothered} implies that $F$ is not of good heredity. \end{proof}

There are many results in the literature about fields such that $K^\times/\mu(K)$ is free Abelian. Here are some non-trivial results we use.

\begin{thm}[May, 1980] \label{may80}
Assume $F$ is a field such that for every finite extension of $E$, $E^\times$ is free modulo torsion, e.g. if $F$ is finitely generated.
\begin{enumerate}
\item If $K$ is any field generated over $F$ by algebraic elements whose degree over $F$ are bounded, then $K^\times$ is free modulo torsion.
\item Suppose, additionally, for every finite extension $E$ of $F$, $\mu(E)$ is finite. Then if $K$ is any Abelian extension of $F$, $K^\times$ is free modulo torsion.
\end{enumerate}
\end{thm}
\begin{proof}
\cite{may-1980}, or \cite{magicbook}, Theorems 10.12 and 10.21.
\end{proof}
\begin{sheep} \label{maysheep}
The maximal Abelian extension of $\Q$ and all of its finite extensions are rootless modtor. The same statement holds for the maximal Abelian extension of any number field. These fields are all of good heredity.
\end{sheep}
\begin{sheep}
For every natural number $k$, let $\Q_k$ be the field obtained by adding the roots of every polynomial with rational coefficients of degree less than or equal to $k$. Then $\Q_k$ and all of its finite extensions are rootless.  The same statement holds for the field obtained by adjoining the $k$-th roots of any set of prime numbers to $\Q$. These fields are all of good heredity.
\end{sheep}

\section{Non-examples}

\subsection{Free modtor doesn't climb.} \label{freemodtor-zample} We note that Warren May has constructed an algebraic extension $F$ of $\Q$ and a quadratic extension $K/F$ with the following properties:
\begin{itemize}
\item $F^*$ is free modtor, but $K^*$ is not free modtor;
\item every finite extension of $F$ contains only finitely many roots of unity.
\end{itemize}
The construction is as follows: Let $\alpha_0 = (2+i)(2-i)^{-1}$, and define a sequence $\alpha_n$ of complex numbers, $n \geq 1$,  by
$$
\alpha_n^4 = \alpha_{n-1}.
$$
Put $K_0 = \Q(i)$, and $K_n = K_0(\alpha_n)$. We then have $K_0 \subset K_1 \subset K_2 \subset \cdots$. We let
$$
K = \bigcup_{n=0}^\infty K_n.
$$
Complex conjugation stabilizes each $K_n$. We let $F = K \cap \R$. Proving that $K$ and $F$ have the desired properties is hard. For the proof, see Theorem 10.18, page 510 of \cite{magicbook}, or the paper \cite{may-1972}.

\subsection{Rootless modtor, but not free modtor}\label{modtormodtor}

One may be tempted to conjecture that fields that are rootless modtor have multiplicative groups that are free modtor. Here we observe that this is not true.
First we recall a result due to Fuchs and Loonstra \footnote{We learned about this result from a post by Andreas Blass on {\tt mathoverflow} on March 8, 2012. We hereby acknowledge this.}:

\begin{thm}[\cite{Fuchs-Loonstra}, Lemma 2]
There exists a torsion free abelian group $G$ of rank $2$ such that every rank $1$ subgroup is cyclic, and every rank $1$ torsion free factor group is divisible.
\end{thm}
In particular this group $G$ is not free.  By Theorem 12.3, page 520 of \cite{magicbook} (Theorem \ref{thm38} below), there is a field $F$ such that
$$
F^\times \simeq G \times {\mathbb{Z}}/ 2{\mathbb{Z}} \times A
$$
with $A$ a free abelian group. It is clear that the torsion subgroup of $F^\times$ is $\Z/2\Z$. Since $F^\times / t(F^\times)$ is isomorphic to $G \times A$, it is clear that $F^\times$ is rootless modtor, and certainly not free modtor.

\subsection{Extensions of rootless (modtor) are not necessarily rootless (modtor)}\label{rottenroots}
In this subsection we construct an example of a finite extension of a rootless (modtor) field which is not rootless (modtor).

\

Let
$$
F_n := \mathbb{Q}(x_n)
$$
with
$$
x_n := \cos(\frac{\alpha}{2^n})
$$
for a real number $\alpha$.  Since there are only countably many reals $\alpha$ such that $\alpha / \pi$ rational and/or $\cos \alpha$ is algebraic, we may and do fix some $\alpha$ such that $\alpha / \pi$ is irrational and $\cos \alpha$ is transcendental. As an abstract field, each $F_n$ is a purely transcendental extension $\mathbb{Q}(y)$ of $\mathbb{Q}$ of transcendence degree $1$.

Let $T(x) := 2x^2 -1$, so that $\cos (2 \theta) = T( \cos \theta)$.
Now $T(x_{n+1}) = x_n$, so the fields $F_n$ form an increasing chain $F_1 \leq F_2 \leq \ldots \leq F_n \leq \ldots$.
As extensions of abstract fields, all pairs $F_{n+1} / F_n$ are isomorphic to $\mathbb{Q}(y) / \mathbb{Q}(T(y))$.
Any two tails $F_\ell \leq F_{\ell+1} \leq \ldots$ and $F_m \leq F_{m+1} \leq \ldots$
of our chain are isomorphic as chains of abstract fields. We finally set
$$F := \bigcup_n F_n.$$

\textit{Claim 1:} The field $F$ is rootless.
\begin{proof}[Proof of Claim 1.] Suppose towards contradiction that $f \in F$ is not a root of unity and has infinitely deep roots in $F$. We may assume without loss of generality that $f \in F_1$, since all tails of the chain are the same.
Since all algebraic elements of $F$ are rational, $f \not\in \mathbb{Q}$. Thus, $f(x_1)$ is a rational function over $\mathbb{Q}$. Since $F_1 = \mathbb{Q}(x_1)$ is rootless, we may and do assume without loss of generality that $f$ has no roots in $F_1$. Since the degree of the field extension $F_n/F_1$ is $2^n$, the infinite chain of roots of $f$ must consist of an infinite chain of square roots. That is, for each $m$ there exists some $n$ and some rational function $g_m$ over $\mathbb{Q}$ such that $(g_m(x_n))^{2^m} = f(x_1)$ in $F_n$. That is,
$$(g_m(y))^{2^m} = f(T^{\circ n} (y)$$
as rational functions in $y$ over $\mathbb{Q}$. Here, $T^{\circ m}$ is the $m$th compositional power of $T$.
All zeros of $(g_m(y))^{2^m}$ on $\mathbb{P}^1$ have multiplicity $r_m 2^m$ for some integer $r_m$. We obtain our contradiction by getting an upper bound on the multiplicity of any root $a$ of $f(T^{\circ n} (y))$ independent of $n$. Indeed, $f$ can only contribute its degree towards the multiplicity of $a$, so we only need an upper bound on the multiplicity of a root $a$ of $T^{\circ n}(y) - T^{\circ n}(a)$. An upper bound on the multiplicities of roots of its derivative suffices. Note that $T^{\circ n}(0) \neq 0$ for any $n$: indeed, $T(0) = -1$, $T(-1) = 1 = T(1)$.
Thus the derivative
$$(T^{\circ n})'(y) = 2 T^{\circ n-1}(y) ( (T^{\circ n-1})'(y)) =
2^n T^{\circ n-1}(y) T^{\circ n-2}(y) \ldots T^{\circ 2}(y) T(y) y$$
has roots of multiplicity at most two, as no two $T^{k}(y)$ share roots because $T^{\circ n}(0) \neq 0$ for any $n$.

To summarize: for any $n$ and any $a$, the multiplicity of roots of $T^{\circ n}(y) - T^{\circ n}(a)$ is at most $3$ because the multiplicity of roots of its derivative is at most $2$. Therefore, for any $n$, the roots of $f(T^{\circ n} (y))$ have multiplicity at most $3 \deg(f)$.
Therefore, for any $m \geq 3 \deg(f)$, there is no rational function over $\mathbb{Q}$ with $(g_m(y))^{2^m} = f(T^{\circ n} (y))$.
That is, there is no $2^m$th root of $f(x_1)$ in $F_n$ for any $n$, which means that $f(x_1)$ has no $2^m$th root anywhere in $F$. \end{proof}

\textit{Claim 1.1:} The field $F$ is rootless modtor.
\begin{proof}[Proof of Claim 1.] Since $F \leq \mathbb{R}$, the only roots of unity in $F$ are $\pm 1$, so the second statement in Corollary \ref{shepherding} makes $F$ rootless modtor.
\end{proof}

\textit{Claim 2:} The quadratic extension $E := F(i\sin(\alpha))$ of the field $F$ is not rootless.
\begin{proof}[Proof of Claim 2.]
Since $$\sin (\frac{\theta}{2}) = \frac{\sin \theta}{2 \cos (\theta/2)},$$
inducting on $n$ shows that $i\sin(\frac{\alpha}{2^n}) \in E$ for each $n$.
 Now $E$ contains complex numbers $c_n := \cos(\frac{\alpha}{2^n}) + i \sin(\frac{\alpha}{2^n})$ with the property that $c_{n+1}^2 = c_n$. Since we chose $\alpha/\pi$ to be irrational at the beginning, these $c_n$ are not roots of unity.
 Thus, $E$ is not rootless.
\end{proof}

\textit{Claim 2.1:} A fortiori, the quadratic extension $E := F(i\sin(\alpha))$ of the field $F$ is not rootless modtor.\\

\textit{Claim 3:} The field $F$ does not have good heredity.
\begin{proof}[Proof of Claim 3.]
Let $P(x) := x^2 - 2 (\cos \alpha) x +1$ be the minimal polynomial of $c_0 = \cos(\alpha) + i \sin(\alpha)$ over $F$.
Over $E$, this polynomial factors as $P(x) = (x + c_0)(x + \overline{c_0})$: since $F \subset \mathbb{R}$, the complex conjugate $\overline{c_0}$ of $c_0$ is the other root of $P$.
Over $E$, the polynomial $x- c_0$ doesn't have good heredity, as $x-c_n$ is a linear factor of $x^{2^n} - c_0$.
Thus, $Q_n(x) := (x + c_n)(x + \overline{c_n}) \in F[x]$ are distinct quadratic factors of $P(x^{2^n})$, showing that $P$ does not have good heredity over $F$.
\end{proof}

\subsection{Rootless, but not rootless modtor}\label{rootless-notrootless}
In this subsection we will construct a field which is rootless but not rootless modtor. We note that this is a group theoretic statement about the multiplicative group of the field.  Let us recall a theorem of W. May \cite{may-1972}.  We call an Abelian group $G$ {\em locally cyclic} if every finitely generated subgroup of $G$ is cyclic.
\begin{thm}[May, 1972]\label{thm38}
Let $G$ be an abelian group such that the torsion subgroup is locally cyclic. Then there is a field $L$ and a group $H$ such that $L^* \simeq G \times H$, where $H$ is a free abelian group if $t(G)$ has non-trivial $2$-component and $H$ is the direct product of a free abelian group with a cyclic group of order $2$ if the torsion subgroup of $G$ has trivial $2$-component.
\end{thm}
This is Theorem 12.3, page 520 of \cite{magicbook}.  Since an Abelian group is rootless, respectively rootless modtor, if and only if its direct product with a free Abelian group is rootless, resp. rootless modtor, May's theorem reduces the problem to finding an Abelian group $G$ with locally cyclic torsion subgroup $t(G)$  such that $G$ is rootless, but $G/t(G)$ is not.  Here, however, instead of construct a concrete field with is rootless, but not rootless modtor.

\

The field $K$ we construct has transcendence degree $1$; its algebraic part $A$ is generated by (some but not all) roots of unity. This field $K$ is an infinite radical extension of $A(x)$ for the transcendental $x$; and this $x$ is not rootless modtor in $K$.

\

First, we build some scaffolding.

\

\textbf{choices:} Let $\{ p_i : i \in \N \}$ be an infinite list of distinct primes; and for each $i$, let $\alpha_i$ be a primitive $p_i$th root of unity.

\

\textbf{non-choices, products:} For each $i$, let $n_i := \prod_{j=0}^{i} p_j$, a square-free integer; and let $\zeta_i := \prod_{j=0}^{i} \alpha_j$, a primitive $n_i$th root of unity.\\

\

\textbf{non-choice, roots: } Since $\gcd (p_{i+1}, n_i) =1$, there are $x,y \in \mathbb{Z}$ such that $xp_{i+1} +y n_{i} = 1$. Letting $\beta_{i+1} := \alpha_{i+1}^y$, we get $\beta_{i+1}^{n_i} = \alpha_{i+1}$. Also, let $\beta_0 = \alpha_0$.

\

Let $A := \mathbb{Q}(\{\alpha_i : i \in \N\})$; by Theorem \ref{may80} and Corollary \ref{maysheep} and Remark \ref{whole-field-ezis}, $A$ is rootless, rootless modtor, and good heredity. We note here for later use that in $A$, a primitive $\ell$th root of unity cannot have an $m$th root unless $\gcd(\ell, m) =1$: there are no $p^2$th roots of unity for any prime $p$.

We build an $\omega$-chain of fields $A(x) \leq A_0 \leq A_1 \leq \ldots$ with each $A_i = A(t_i)$ for transcendentals $t_i$. The goal of the construction is to make $x$ not rootless modtor in $K := \cup_i A_i$, while keeping everything rootless.

We embed $A(x)$ into $A(t_0)$ over $A$ by sending $x$ to $\beta_0 t_0^{p_0}$.
We embed $A_i$ into $A_{i+1}$ over $A$ by sending $t_i$ to $\beta_{i+1} t_{i+1}^{p_{i+1}}$.

\

\textbf{Claim 1:} In $K$, $x = \zeta_i t_i^{n_i}$ for all $i$.
\begin{proof}[Proof of Claim 1] The base case for the induction on $i$ is our choice of the embedding of $A(x)$ into $A(t_0)$ and our choice of $\beta_0$. For the induction step,
$$\zeta_i t_i^{n_i} = \zeta_i ( \beta_{i+1} t_{i+1}^{p_{i+1}} )^{n_i} =
(\zeta_i \beta_{i+1}^{n_i} ) t_{i+1}^{ p_{i+1} n_i } =
(\zeta_i \alpha_{i+1} ) t_{i+1}^{n_{i+1}} = \zeta_{i+1} t_{i+1}^{n_{i+1}}.$$

In particular, $x$ is an $n_i$th power modulo roots of unity, for every $i$, so it is not rootless modtor; so the field $K$ is not rootless modtor.
\end{proof}

\textbf{Claim 2:} The field $K$ is rootless.
\begin{proof}[Proof of Claim 2] As with \S \ref{rottenroots}, it suffices to show that elements of $A(x)$ are rootless in $K$. While it is not true that all tails of the chain $A_i$ are isomorphic, the only difference is the particular primes $p_i$ and a few extra roots of unity coprime to everything we care about.
\end{proof}

As in \S\ref{rottenroots}, we take an element $f(x) \in A(x)$ and find a bound on $m$ such that $f(x) \in A_i^m$ independent of $i$.

\

\textit{Case ``constant'':} If $f \in A$ is not a root of unity, then $f$ is rootless in $A$ as we noted above.
Since each $A_i = A(t_i)$ is a purely transcendental extension of $A$, $f$ gains no new roots in any $A_i$.\\

\

\textit{Case ``power'':} Suppose that $f = x^k$ for some nonzero $k \in \mathbb{Z}$.
In $A_i = A(t_i)$ we have from Claim 1 that $x = \zeta_i t_i^{n_i}$, so
 $$f(x) = f(\zeta_i t_i^{n_i}) = \zeta_i^k t_i^{k n_i} =: g(t_i).$$
 Factors of monomials must be monomials, so if $g = h^m$ for some $m$ and some $h \in A_i$, then
 $h(t_i) = b t_i^r$ and $rm = k n_i$ and $b^m = \zeta_i^k$. We are now entirely inside the torsion group of $A$.

 Recall that $\zeta_i$ is a primitive $n_i$th root of unity, so $\eta_i := \zeta_i^k$ is a primitive $\ell_i$th root of unity where $\ell_i := (n_i/\gcd(n_i,k))$. In order for  $\eta_i$ to have an $m$th root, we must have $\gcd(\ell_i, m) =1$. Now we are just solving divisibility-and-gcd relations in integers.

 Let $\gamma := \gcd(k, n_i)$ and $n_i = \gamma \ell_i$ and $k = \gamma k'$ with $\gcd(k', \ell_i) =1$.
 We also know that $m$ divides $kn_i = \gamma^2 k' \ell_i$ and that $\gcd(\ell_i, m) =1$.
 Thus, $m$ must divide $\gamma^2 k'$; so $m$ must divide $k^2$. Actually, $m$ must also be square-free, so it must actually divide $k$. In any case, $m$ is bounded independently of $i$.

\

\textit{Case ``monomial'':} Suppose that $f = ax^k$ for some $a \in A$ and some nonzero $k \in \mathbb{Z}$.

 We first show that we may assume without loss of generality that $a=1$.
In $A_i = A(t_i)$ we have from Claim 1 that $x = \zeta_i t_i^{n_i}$, so
 $$f(x) = f(\zeta_i t_i^{n_i}) = a \zeta_i^k t_i^{k n_i} =: g(t_i).$$
 Factors of monomials must be monomials, so if $g = h^m$ for some $m$ and some $h \in A_i$, then
 $h(t_i) = b t_i^r$ and $rm = k n_i$ and $b^m = a \zeta_i^k$.
 Unless $a$ is a root of unity, the fact that $A$ is rootless modtor gives the desired bound on $m$.
 So suppose that $a^q =1$; now it suffices to show that $f(x)^q = x^{kq}$ is rootless, which was done in the previous Case.

\

\textit{Case ``other'':} Otherwise, some nonzero $\alpha \in A^{alg}$ is a zero or a pole of $f$. Since the rootlessness of $f$ is equivalent to the rootlessness of $1/f$, we may and do assume without loss of generality that $f(\alpha) = 0$. Let $r$ be the multiplicity of $\alpha$ as a zero of $f$.
Now in $A_i = A(t_i)$ we have from Claim 1 that $x = \zeta_i t_i^{n_i}$, so $f(x) = f(\zeta_i t_i^{n_i}) =: g(t_i)$.
So now for each of the $n_i$ distinct $n_i$th roots $\beta$ of $\alpha/\zeta_i$ in $A^{alg}$ is a zero of $g$ with multiplicity exactly $r$. On the other hand, if $f(x) = g(t_i) = (h(t_i))^m$ for some $h \in A_i$, then $m$ divides the multiplicity of all zeros of $g$. Thus, $r$ is an upper bound on $m$, independent of $i$.

\end{document}